\documentclass{amsart}
\pdfoutput=1
\usepackage{times, enumerate, array, longtable, %
amssymb, diagrams, mathrsfs, bbm, times, txfonts, verbatim}
\usepackage[initials]{amsrefs}

\title{Prime Filters in MV-algebras II}
\author{Colin G. Bailey}
\address{School of Mathematics,  Statistics \& Operations Research\\
Victoria University of Wellington\\
Wellington, New Zealand\\
}
\email{Colin.Bailey@vuw.ac.nz}

\date{\today}

\def\one{{\mathbf 1}}
\def\zero{{\mathbf 0}}

\def\eqcl[#1]{\pmb{[}#1\pmb{]}}

\def\leftGen{{[\kern-1.1pt[}}
\def\rightGen{{]\kern-1.1pt]}}

\let\sqTo\rightsquigarrow
\let\rsf\mathscr
\let\dcl\downarrow
\providecommand{\meet}{\mathbin{\wedge}}
\providecommand{\join}{\mathbin{\vee}}

\newcommand{\card}[1]{\left| #1\right|}
\ifx\upharpoonright\undefined
     \def\restrict{\hbox{\rm\kern0.166em\accent"12\kern-0.536em$\vert$\kern0.3em}}%
  \else
     \def\restrict{\upharpoonright}%
  \fi
\makeatletter
\def\twoSet#1#2{\left\{%
\vphantom{#2}#1\thinspace\right|\nolinebreak[3]\left.%
  #2%
  \vphantom{#1}%
  \right\}%
}
\def\oneSet#1{\left\lbrace#1\right\rbrace}

\newif\if@nstr
\def\setstrfalse{\let\if@nstr=\iffalse}
\def\setstrtrue{\let\if@nstr=\iftrue}
\def\@nstr #1#2{
\def\@@nstr ##1#1##2##3\@@nstr{\ifx
\@nstr ##2\setstrfalse \else \setstrtrue \fi }
\@@nstr #2#1\@nstr \@@nstr}
\def\@separate#1|#2@{\setFront{#1}\setBack{#2}}
\def\lb#1\rb{\@nstr|{#1} \if@nstr \@separate#1 @ \twoSet{\@setFront}{\@setBack}%
\else \@separate |{#1 }@ \oneSet{\@setBack}\fi%
}
\def\setFront#1{\def\@setFront{#1}}
\def\setBack#1{\def\@setBack{#1}}
\def\Set#1{\lb{#1}\rb}

\def\oneBrk#1{\left\langle#1\right\rangle}
\def\twoBrk#1#2{\left\langle%
\vphantom{#2}#1\thinspace\right|\nolinebreak[3]\left.%
  #2%
  \vphantom{#1}%
  \right\rangle%
}
\def\brk<#1>{\@nstr|{#1} \if@nstr \@separate#1 @ \twoBrk{\@setFront}{\@setBack}%
\else \@separate |{#1 }@ \oneBrk{\@setBack}\fi%
}

\def\lemref#1{\normalfont{lemma}~\ref{#1}}

\def\propref#1{\normalfont{proposition}~\ref{#1}}
\theoremstyle{plain}
\newtheorem{thm}{Theorem}[section]
\newtheorem{lem}[thm]{Lemma}
\newtheorem{cor}[thm]{Corollary}
\newtheorem{prop}[thm]{Proposition}
\newtheorem{defn}[thm]{Definition}

\theoremstyle{remark}
{}
{}
{}
{}

\@ifpackageloaded{bbm}{
\newcommand{\Q}{{\mathbbm{Q}}}

\newcommand{\R}{{\mathbbm{R}}}

}{

\newcommand{\Q}{{\mathbb{Q}}}

\newcommand{\R}{{\mathbb{R}}}
}
\makeatother

\begin{document}
\begin{abstract} 	
    In this document we consider the prime and maximal spectra of 
   an MV-algebra with certain natural operations. Several new 
   MV-algebras are constructed in this fashion.
\end{abstract}
\maketitle

\section{Introduction}
In the representation theorems for MV-algebras of Martinez 
(\cite{Martinez:one, Martinez:two}) and Martinez \& Priestley 
(\cite{MartPreist:one}) much use was made of 
the function on filters: $a\mapsto \rsf F_{a}=\Set{z | z\to a\notin\rsf F}$. For 
a fixed lattice filter $\rsf F$ the set $\Set{\rsf F_{a} | 
a\in\mathcal L}$ was shown to be linearly ordered. In 
a previous paper \cite{mvpaperOne} we defined the kernel of a filter
and showed how to compute the kernel of several natural filters.
In this paper we extend this analysis by generalising the definition 
of kernel. We are able to show that the class of prime lattice 
filters of an MV-algebra naturally decomposes into linearly ordered 
MV-algebras. We also relate these operations to those defined in 
\cite{Martinez:one, Martinez:two} and to other natural operations. 

\begin{defn}
    Let $\rsf F$ be any order filter.
    The set 
    $$
    \rsf F_{a}=\Set{z | z\to a\notin\rsf F}
    $$
    is called the \emph{subordinate of $\rsf F$ at $a$}.
\end{defn}

We will primarily be considering the case when $\rsf F$ is prime
which implies that $\rsf F_{a}$ is also a prime filter. 

\begin{defn}
    The \emph{kernel} of an order filter $\rsf F$ is the set
    $$
    \mathcal{K}(\rsf F)=\Set{z | \forall a\notin\rsf F\ z\to 
    a\notin\rsf F}.
    $$
    
    The \emph{prime spectrum} of a prime implication filter $P$ is
    $$
    \text{PSpec}(P)=\Set{\rsf F | \mathcal K(\rsf F)=P}.
    $$
\end{defn}

We distinguish the  special case of $a=0$ with the notation:
$$
\rsf F^{+}=\rsf F_{0}.
$$

$\mathcal K(\rsf F)$ is an implication filter contained in $\rsf F$ 
and is prime iff $\rsf F$ is prime.
In \cite{mvpaperOne} we showed that 

\begin{prop}\label{prop:SubAEq}
    If $a\notin\rsf F$ then $\mathcal K(\rsf F)=\mathcal K(\rsf F_{a})$.
\end{prop}

The definition of $\mathcal K(\rsf F)$ can clearly be generalised to 
the following:
\begin{defn}
    If $\rsf F$ is any prime lattice filter and $X\cap\rsf F=\emptyset$ let
    $$
    \mathcal K(\rsf F; X)=\bigcap_{a\in X}\rsf F_{a}.
    $$
\end{defn}

Clearly we have -- if $\rsf K=\mathcal K(\rsf F)$ then
\begin{enumerate}[(a)]
    \item  $\mathcal K(\rsf F;  X)$ is a filter -- as it is the 
    intersection of filters.
    
    \item  $\rsf F_{a}=\mathcal K(\rsf F; \Set a)$ for $a\notin\rsf F$; 

    \item  $\mathcal K(\rsf F)=\mathcal K(\rsf F; \mathcal 
    L\setminus\rsf F)$; 

    \item  $\mathcal K(\rsf F;  X)$ depends only on $X/\mathcal 
    K(\rsf F)$ -- as
    $\eta(a)=\eta(b)$ iff $\rsf F_{a}=\rsf F_{b}$.
    
    \item if $X\dcl=\Set{ z | \exists x_{1}, \dots,  x_{n}\in X\ z\le 
    \bigvee_{i}x_{i}}$ 
    then $\mathcal K(\rsf F; X)=\mathcal K(\rsf F; 
    X\dcl)$ -- since $z\le x$ implies $\rsf 
    F_{x}\subseteq\rsf F_{z}$ and $\rsf F_{a\join b}$ is always either 
    $\rsf F_{a}$ or $\rsf F_{b}$.
\end{enumerate}

Because of part (e) we need only consider $\mathcal K(\rsf F; I)$ when 
$I$ is a lattice ideal. A particular case gives us the following 
definition.
\begin{defn}
    Let $\rsf F$ and $\rsf G$ be two prime lattice filters with 
    $\rsf F\subseteq\rsf G$. Then
    $$
    \rsf F\sqTo\rsf G=\mathcal K(\rsf F; \mathcal L\setminus\rsf G).
    $$
\end{defn}
We extend this definition to all $\rsf F$ and $\rsf G$ by 
defining
$$
    \rsf F\sqTo\rsf G=(\rsf F\cap\rsf G)\sqTo\rsf G.
$$
We intend to show that there are naturally defined sets of prime 
filters closed under this operation,  that are MV-algebras,  with 
order being reverse inclusion. 

\section{Basic facts}
Here are some of the easier facts about this operation.

\begin{prop}
    Let $\rsf F\subseteq\rsf G$. Then
    $$
	\rsf F\sqTo\rsf G=\Set{z | \forall f\in\rsf F\ f\otimes 
	z\in\rsf G}.
    $$
\end{prop}
\begin{proof}
    \begin{align*}
        z\in\rsf F\sqTo\rsf G & \iff \forall a\notin\rsf G\ z\to 
	a\notin\rsf F  \\
         & \iff \forall a\notin\rsf G\ \forall f\in\rsf F\lnot(f\le 
	 z\to a)\\
         & \iff \forall a\notin\rsf G\ \forall f\in\rsf 
	 F\lnot(f\otimes z\le a) \\
         & \iff \forall f\in\rsf F\ z\otimes f\notin\mathcal 
	 L\setminus\rsf G  \\
         & \iff\forall f\in\rsf F\ z\otimes f\in\rsf G.
    \end{align*}
\end{proof}

\begin{prop}\label{prop:incl}
    $\rsf F\sqTo\rsf G\subseteq\rsf G$.
\end{prop}
\begin{proof}
    As if $z\in\rsf F\sqTo\rsf G$ then $1\in\rsf F$ and so 
    $z=1\otimes z\in\rsf G$.
\end{proof}

\begin{prop}\label{prop:inclOne}
    If $\rsf G\subseteq\rsf F$ then $\rsf F\sqTo\rsf G= \mathcal 
    K(\rsf G)$.
\end{prop}
\begin{proof}
    As $\rsf F\sqTo\rsf G= (\rsf F\cap\rsf G)\sqTo\rsf G= \rsf 
    G\sqTo\rsf G= \mathcal 
    K(\rsf G)$.
\end{proof}

\begin{prop}
    Let $\rsf F\subseteq\rsf G_{1}\subseteq\rsf G_{2}$. Then
    $$
	\rsf F\sqTo\rsf G_{1}\subseteq\rsf F\sqTo\rsf G_{2}.
    $$
\end{prop}
\begin{proof}
    Let $z\in\rsf F\sqTo\rsf G_{1}$ and $f\in\rsf F$. Then $z\otimes 
    f\in\rsf G_{1}$ and so must be in $\rsf G_{2}$.
\end{proof}

\begin{prop}\label{prop:revIncl}
    Let $\rsf F_{1}\subseteq\rsf F_{2}\subseteq\rsf G$. Then
    $$
	\rsf F_{2}\sqTo\rsf G\subseteq\rsf F_{1}\sqTo\rsf G.
    $$
\end{prop}
\begin{proof}
    Let $z\in\rsf F_{2}\sqTo\rsf G$ and $f\in\rsf F_{1}$. Then 
    $f\in\rsf F_{2}$ and so $z\otimes f\in\rsf G$.
\end{proof}

\begin{prop}\label{prop:plus}
    Let $\rsf F\subseteq\rsf G$. Then
    $$
	\rsf F\sqTo\rsf G=\rsf G^{+}\sqTo\rsf F^{+}.
    $$
\end{prop}
\begin{proof}
    Let $z\in\rsf F\sqTo\rsf G$ and $a\notin\rsf F^{+}$. Then $\lnot 
    a\in\rsf F$ and so $z\otimes\lnot a=\lnot (z\to a)\in\rsf G$. 
    Hence $z\to a\notin\rsf G^{+}$. 
    
    Thus we have $\rsf F\sqTo\rsf G\subseteq\rsf G^{+}\sqTo\rsf F^{+}$. 
    
    And so $\rsf G^{+}\sqTo\rsf F^{+}\subseteq \rsf F^{++}\sqTo\rsf 
    G^{++}= \rsf F\sqTo\rsf G$.
\end{proof}

\begin{prop}\label{prop:OneOne}
    $$
	\mathcal K(\rsf F)\sqTo\rsf F=\rsf F.
    $$
\end{prop}
\begin{proof}
    We have $\mathcal K(\rsf F)\sqTo\rsf F\subseteq\rsf F$ from 
    \propref{prop:incl}.
    
    If $f\in\rsf F$ then for any $k\in\mathcal K(\rsf F)$ we have 
    $f\otimes k\in\rsf F$. Hence $f\in\mathcal K(\rsf F)\sqTo\rsf F$.
\end{proof}

\begin{prop}
    Let $\rsf F, \rsf H\subseteq\rsf G$. Then
    $$
    \rsf F\subseteq\rsf H\sqTo\rsf G \text{ iff }\rsf H\subseteq\rsf 
    F\sqTo\rsf G.
    $$
\end{prop}
\begin{proof}
    We note that
    $\rsf F\subseteq\rsf H\sqTo\rsf G$ iff
    $\forall f\in\rsf F\ \forall h\in\rsf H\ f\otimes h\in\rsf G$
    is symmetric in $\rsf F$ and $\rsf H$.
\end{proof}

\begin{prop}\label{prop:axiomC}
    $$
	\rsf F\sqTo(\rsf H\sqTo\rsf G)=\rsf H\sqTo(\rsf F\sqTo\rsf G).
    $$
\end{prop}
\begin{proof}
    We may assume that $\rsf F, \rsf H\subseteq\rsf G$. Then we see 
    that
    $x\in\rsf F\sqTo(\rsf H\sqTo\rsf G)$ iff $\forall f\in\rsf F\ 
    \forall h\in\rsf H\ (x\otimes f)\otimes h\in\rsf G$. As
    $\otimes$ is associative,  this is symmetric in $\rsf F$ and 
    $\rsf H$ and so we have equality.
\end{proof}

\subsection{Martinez' $\Phi$ function}\label{secsec:Phi}
In \cite{Martinez:two} Martinez defined the binary function on lattice filters
$$
\Phi(\rsf F, \rsf G)=\bigcup_{f\in\rsf F}\Set{y | f\to y\in\rsf G}.
$$
We note that this can be defined in terms as $\sqTo$ and 
$\bullet^{+}$ as $(\rsf F\sqTo\rsf G^{+})^{+}$, as we have
\begin{align*}
    z\in\Phi(\rsf F, \rsf G) & \text{ iff }\exists f\in\rsf F\ f\to 
    y=\lnot f\oplus y\in\rsf G  \\
     & \text{ iff }\lnot(\forall f\in\rsf F\ \lnot f\oplus y\notin\rsf G)  \\
     & \text{ iff }\lnot(\forall f\in\rsf F\ f\otimes\lnot y \in\rsf 
     G^{+})  \\
     & \text{ iff }\lnot(\lnot z\in\rsf F\sqTo\rsf G^{+})  \\
     & \text{ iff }\lnot( z\notin(\rsf F\sqTo\rsf G^{+})^{+})  \\
     & \text{ iff } z\in(\rsf F\sqTo\rsf G^{+})^{+}).
\end{align*}

We note that in any MV-algebra $\lnot(f\to\lnot g)= f\otimes g$ and 
so, in some sense, we have 
$$
\Phi(\rsf F, \rsf G)=\rsf F\otimes\rsf G.
$$

\section{Sensitivity to Kernels}
The definition of $\sqTo$ exhibits some interesting sensitivity to kernels of 
the filters involved as the next couple of results show us.

For a given implication filter $P$ we let 
$\eta_{P}\colon\mathcal L\to\mathcal L/P$
denote the canonical epimorphism. 
We usually omit the subscript when $P$ is clear from context.

\begin{defn}
    Let $\rsf F$ be any filter and $P$ any implication filter. Then
    $$
    J_{u}(\rsf F,  P)=\eta_{P}^{-1}[\rsf F/P].
    $$
\end{defn}

Clearly $\rsf F\subseteq J_{u}(\rsf F,  P)$ and $J_{u}$ is prime or 
trivial.
In \cite{mvpaperOne} we showed that if $P$ is prime and $J_{u}$ is 
nontrivial then 
$$
    \mathcal K(J_{u}(\rsf F,  P))=\mathcal K(\rsf F)\join P.
$$

\begin{prop}\label{prop:small}
    $J_{u}(\rsf F,  P)$ is the smallest filter $\rsf H$ satisfying:
    \begin{enumerate}[(a)]
        \item  $\rsf F\subseteq\rsf H$; 
    
        \item  $P\subseteq\mathcal K(\rsf H)$.
    \end{enumerate}
\end{prop}
\begin{proof}
    Let $\rsf H$ have the two properties listed. Then 
    $\rsf F/P\subseteq\rsf H/P$ and so 
    $J_{u}(\rsf F, P)= \eta_{P}^{-1}[\rsf F/P]\subseteq \eta_{P}[\rsf 
    H/P]= \rsf H$ (as $P\subseteq\mathcal K(\rsf H)$).
\end{proof}

We also need the dual notion
$$
	J_{d}(\rsf F,  P)=\left(\eta_{P}^{-1}[\rsf F^{+}/P]\right)^{+}.
$$
As  $\rsf F^{+}\subseteq J_{u}(\rsf F^{+},  P)$ we have 
$J_{d}(\rsf F,  P)= J_{u}(\rsf F^{+},  P)^{+}\subseteq \rsf F$. 
However it may be empty. There are certain situations where $J_{d}$
produces nontrivial filters (which must be prime). 

\begin{lem}
    Let $\rsf F\subseteq \rsf G$ be two prime filters. Then 
    $$
	\rsf F\subseteq J_{d}(\rsf G, \mathcal K(\rsf F)).
    $$
\end{lem}
\begin{proof}
    Let $\rsf K=\mathcal K(\rsf F)$.
    
    $\rsf F\subseteq\rsf G$ implies $\rsf G^{+}\subseteq\rsf F^{+}$
    and so 
    $$
	\eta_{\rsf K}^{-1}[\rsf G^{+}/\rsf K]\subseteq\eta_{\rsf 
	K}^{-1}[\rsf F^{+}/\rsf K]=\rsf F^{+}
    $$
    as $\mathcal K(\rsf F^{+})=\mathcal K(\rsf F)$,  by \propref{prop:SubAEq}. 
    Hence 
    $\rsf F\subseteq  J_{d}(\rsf G, \mathcal K(\rsf F))$.
\end{proof}

\begin{prop}
    $J_{d}(\rsf F,  P)$ is the largest filter $\rsf H$ satisfying:
    \begin{enumerate}[(a)]
        \item  $\rsf H\subseteq\rsf F$; 
    
        \item  $P\subseteq\mathcal K(\rsf H)$.
    \end{enumerate}
\end{prop}
\begin{proof}
    This follows from \propref{prop:small} as $\rsf H\subseteq\rsf F$ 
    iff $\rsf F^{+}\subseteq\rsf H^{+}$.
\end{proof}

\begin{prop}
    Let $\rsf F\subseteq\rsf G$ be two prime lattice filters. Then
    $$
	\rsf F\sqTo\rsf G=J_{u}(\rsf F,  \mathcal K(\rsf G))\sqTo\rsf 
	G.
    $$
\end{prop}
\begin{proof}
    The right-to-left inclusion follows from $\rsf F\subseteq J=J_{u}(\rsf F,  \mathcal K(\rsf G))$.
    
    Conversely,  if $z\in \rsf F\sqTo\rsf G$ and $j\in J$ we want to 
    show that $z\otimes j\in\rsf G$. 
    
    There is some $f\in \rsf F$ with $f\sim j\mod\mathcal K(\rsf 
    G)$,  and so $z\otimes f\sim z\otimes j$. As $z\otimes f\in\rsf G$
    and $\rsf G$ is closed under $\sim_{\rsf K}$ we have $z\otimes 
    j\in\rsf G$. 
\end{proof}

\begin{cor}
    Let $\rsf F\subseteq\rsf G$ be two prime lattice filters. Then
    $$
	\rsf F\sqTo\rsf G=\rsf F\sqTo J_{d}(\rsf G,  \mathcal K(\rsf F)).
    $$
\end{cor}
\begin{proof}
    \begin{align*}
        \rsf F\sqTo\rsf G & =\rsf G^{+}\sqTo\rsf F^{+}  \\
         & = J_{u}(\rsf G^{+}, \mathcal K(\rsf F^{+}))\sqTo\rsf F^{+}  \\
         & = \rsf F\sqTo J_{u}(\rsf G^{+}, \mathcal K(\rsf 
	 F^{+}))^{+}\\
	 & = \rsf F\sqTo J_{d}(\rsf G,  \mathcal K(\rsf F))
    \end{align*}
    as $\mathcal K(\rsf F)=\mathcal K(\rsf F^{+})$.
\end{proof}

\begin{thm}
    Let $\rsf F\subseteq\rsf G$ be two prime lattice filters. Then
    $$
	\rsf F\sqTo\rsf G=J_{u}(\rsf F,  \mathcal K(\rsf G)) \sqTo J_{d}(\rsf G,  \mathcal K(\rsf F)).
    $$
\end{thm}
\begin{proof}
    We note that $\mathcal K(J_{d}(\rsf G, \mathcal K(\rsf 
    F)))=\mathcal K(\rsf G)\join\mathcal K(\rsf F)$ and so 
    $$
	J_{d}(\rsf G, \mathcal K(\rsf F))= J_{d}(\rsf G,  \mathcal 
	K(\rsf G)\join\mathcal K(\rsf F)).
    $$
    Likewise
    $$
	J_{u}(\rsf F, \mathcal K(\rsf G))= J_{u}(\rsf F,  \mathcal 
	K(\rsf G)\join\mathcal K(\rsf F)).
    $$
    Hence we have 
    \begin{align*}
        \rsf F\sqTo\rsf G & = J_{u}(\rsf F,  \mathcal K(\rsf 
	G))\sqTo\rsf G  \\
         & = J_{u}(\rsf F,  \mathcal K(\rsf G))\sqTo J_{d}(\rsf G,  
	 \mathcal K(J_{u}(\rsf F,  \mathcal K(\rsf G))))  \\
         & = J_{u}(\rsf F,  \mathcal K(\rsf G))\sqTo J_{d}(\rsf G,  
	 \mathcal K(\rsf F)\join\mathcal K(\rsf G))\\
	 &= J_{u}(\rsf F,  \mathcal K(\rsf G))\sqTo J_{d}(\rsf G,  
	 \mathcal K(\rsf F)).
    \end{align*}
\end{proof}

From this theorem we see that we can largely reduce our study of 
$\sqTo$ to the case when the 
two filters involved have the same kernel. 

The next step is to compute the kernel of $\rsf F\sqTo\rsf G$ when 
$\rsf F$ and $\rsf G$ have the same kernel. We begin this by 
considering the interaction between $\sqTo$ and quotients.

\section{Interaction with quotients}
\begin{prop}
    Let $\rsf F\subseteq\rsf G$ be two lattice filters,  and $Q$ any 
    implication filter with $Q\subseteq\mathcal 
    K(\rsf G)$. Then
    $$
	(\rsf F\sqTo\rsf G)/Q= (\rsf F/Q)\sqTo(\rsf G/Q).
    $$
\end{prop}
\begin{proof}
    Let $[z]\in(\rsf F\sqTo\rsf G)/Q$ with $z\in\rsf F\sqTo\rsf G$. 
    Then for any $[y]\in\rsf F/Q$ we can assume that $y\in\rsf F$ and
    so $z\otimes y\in\rsf G$. Thus $[z]\otimes[y]= [z\otimes 
    y]\in\rsf G/Q$.
    
    Conversely, if $[z]\in (\rsf F/Q)\sqTo(\rsf G/Q)$ and $y\in\rsf 
    F$,  then $[y]\in\rsf F/Q$ and so $[z\otimes y]\in\rsf G/Q$. 
    Since $Q\subseteq\mathcal K(\rsf G)$ we know that $\rsf G$ is 
    closed under $Q$-cosets and so $z\otimes y\in\rsf G$. 
\end{proof}


\begin{prop}\label{prop:quot}
    $\rsf F\subseteq\rsf G$ be two prime lattice filters,  and $Q$ any 
    implication filter with $Q\subseteq\mathcal 
    K(\rsf G)=\mathcal K(\rsf F)$. Then
    $$
    	\eta_{Q}^{-1}[\rsf F/Q\sqTo\rsf G/Q]=\rsf F\sqTo\rsf G.
    $$
\end{prop}
\begin{proof}
    Note that $\eta^{-1}[\rsf F/Q]=\rsf F$ and 
    $\eta^{-1}[\rsf G/Q]=\rsf G$ as $Q\subseteq\mathcal K(\rsf 
    F)=\mathcal K(\rsf G)$.
    
    Let $z\in\eta^{-1}[\rsf F/Q\sqTo\rsf G/Q]$, so that
    $[z]_{Q}\in \rsf F/Q\sqTo\rsf G/Q$. Choose $h\notin\rsf G$. As
    $Q\subseteq\mathcal K(\rsf G)$ this implies $[h]_{Q}\notin\rsf G/Q$
    and so $[z]_{Q}\to[h]_{Q}=[z\to h]_{Q}\notin\rsf F/Q$. As
    $Q\subseteq\mathcal K(\rsf F)$ we have $z\to h\notin\rsf F$.
    Hence $z\in\rsf F\sqTo\rsf G$.
    
    Conversely,  if $z\in\rsf F\sqTo\rsf G$ and $[h]_{Q}\notin\rsf 
    G/Q$ then $h\notin\rsf G$,  so that $z\to h\notin\rsf F$. Hence 
    $[z]_{Q}\to[h]_{Q}=[z\to h]_{Q}\notin\rsf F/Q$. Thus $[z]_{Q}\in 
    \rsf F/Q\sqTo\rsf G/Q$.
\end{proof}

\begin{cor}
    $\rsf F\subseteq\rsf G$ be two prime lattice filters with 
    $\mathcal K(\rsf G)=\mathcal K(\rsf F)$. Then
    $$
    	\mathcal K(\rsf F)\subseteq\mathcal K(\rsf F\sqTo\rsf G).
    $$
\end{cor}
\begin{proof}
    We know that $P\subseteq\mathcal K(\rsf H)$ iff 
    $\eta_{P}^{-1}[\rsf H/P]=\rsf H$ for any filter $\rsf H$.
    Let $\rsf K=\mathcal K(\rsf F)$. Then   
    from above we have 
    $
    \rsf F\sqTo\rsf G= \\
    \eta_{\rsf K}^{-1}[\rsf F/\rsf K\sqTo\rsf 
    G/\rsf K]= \eta_{\rsf K}^{-1}[(\rsf F\sqTo\rsf G)/\rsf K].
    $
\end{proof}

\section{The kernel of $\rsf F\sqTo\rsf G$}

\begin{thm}
	Let $\rsf F\subseteq\rsf G$ be two prime filters with $\mathcal K(\rsf F)=\mathcal K(\rsf G)$.
	Then
	$$
		\mathcal K(\rsf F\sqTo\rsf G)=\mathcal K(\rsf F).
	$$
\end{thm}
\begin{proof}
	Let $\rsf K=\mathcal K(\rsf F)$.
	
	Let $\rsf K\subsetneq P$ for some implication filter $P$. We show that
	 $\mathcal K(\rsf F\sqTo\rsf G)$ is properly contained in 
	 $P$. As we already know that $\mathcal K(\rsf 
	 F)\subseteq\mathcal K(\rsf F\sqTo\rsf G)$ this gives the 
	 result.
	 
	 As $P$ properly extends $\rsf K$ there
	 are $P$-cosets $S,T$ such that
	 \begin{align*}
		T\cap\rsf F\not=\emptyset &\not=T\setminus\rsf F\\
		\text{ and }
		S\cap\rsf G\not=\emptyset &\not=S\setminus\rsf G. 
	 \end{align*} 
	 Pick $b\in S\setminus\rsf G$ and $a\in T\cap\rsf F$. Then
	 $a\to b\notin\rsf F\sqTo\rsf G$ as $b\notin\rsf G$ but 
	 $(a\to b)\to b= a\join b\in\rsf F$. 
	 
	 Now pick $b'\in S\cap\rsf G$ and $a'\in T\setminus\rsf F$. 
	 
	 Now either $b'<_{\rsf K}a'$ or $a'<_{\rsf K}b'$. 
	 
	 In the former case we consider 
	 that any $c<_{\rsf K}\rsf G$ then $c<_{\rsf K} b'$ and we have 
	 $(a'\to b')\to c'<_{\rsf K} (a'\to b')\to b'= a'\join b'\sim_{\rsf K}a'\notin\rsf F$.
	 Hence $(a'\to b')\to c\notin\rsf F$ and so $a'\to b'\in\rsf F\sqTo\rsf G$.
	 
	 As $a\to b$ and $a'\to b'$ are in the same $P$-coset we see that $\rsf F\sqTo\rsf G$
	 is not closed under $\sim_P$.
	 
	 Now the other case -- we are forced to take $a'<_{\rsf 
	 K}b'$. Therefore
	 $b'=a'\join b'>a'$ as $a'<_{\rsf K} b'$.
	 Then we get $a'\to b'=1\in\rsf F\sqTo\rsf G$ and see that we have a $P$-coset that
	 is both in and out of $\rsf F\sqTo\rsf G$. 
\end{proof}

Because this always works we may assume that $\rsf K=\Set1$ and we are working in
a linearly ordered MV-algebra. This simplifies things a lot. However 
the case where $\mathcal L$ is discrete is very easy to understand.
Before we consider this case we need some technical results on 
convexity.  

\section{Convexity}
In what follows we often need to find an element of the algebra with 
certain order properties that also satisfies some equation. In order to do 
this we often appeal to a convexity argument that works provided we 
know suitable bounds and that sets of interest are convex. To that 
end we have the following technical facts.

\begin{lem}
    If $C$ is a convex set and $a\in\mathcal L$ then $\Set{z\to a | 
    z\in C}$ is convex.
\end{lem}
\begin{proof}
    There are three 
    cases:
    \begin{enumerate}[{Case }1:]
	\item  $C<a$ in which case the set contains only $1$.
    
	\item  $a<C$. Let $p\to a< y < q\to a$ for some $p, q\in C$. 
	Then $q=(q\to a)\to a\le y\to a\le p=(p\to a)\to a$ and so 
	$y\to a\in C$. Hence $y=(y\to a)\to a$.
    
	\item  $a\in C$. So let $p>a$ and $p\to a< y< 1=a\to a$ -- 
	any other possibility is like the last case. 
	Then 
	$1\to a=a \le y\to a\le (p\to a)\to a=p$ so again $y\to a\in C$. 
	If $y\to a=1$ then $a\le p\to a< y\le a$ which is impossible.
	
	Hence $y=(y\to a)\to a$ is in the set.
    \end{enumerate}    
\end{proof}

The next lemma is immediate. 
\begin{lem}
    If $C$ is a convex set and $a\in\mathcal L$ then $\Set{\lnot z | 
    z\in C}$ is convex.
\end{lem}

\begin{lem}
    If $C$ is a convex set and $a\in\mathcal L$ then $\Set{z\otimes a | 
    z\in C}$ is convex.
\end{lem}
\begin{proof}
    $z\otimes a= \lnot (z\to\lnot a)$ so we can combine the last two 
    results to get this.
\end{proof}

\section{The discrete case}
\begin{defn}
    A \emph{discrete point} is a point $a$ with an immediate 
    successor or an immediate predecessor.
\end{defn}

It is easy to show the following propositions.
\begin{prop}
    If $\mathcal L$ has a discrete point then $0$ has a successor $c$.
\end{prop}

\begin{prop}
    If $0$ has successor $c$ then $a\oplus c$ is the successor of 
    every $a<1$ and $a\ominus c$ is the predecessor of every $a>0$. 
\end{prop}

\begin{thm}
    If $\mathcal L$ is discrete and $\rsf F$ is a filter with 
    $\mathcal K(\rsf F)=\Set1$ then $\rsf F$ is principal.
\end{thm}
\begin{proof}
    If $\rsf F$ is nonprincipal, and $a\in\rsf F$ then we must have 
    $a\ominus c\in\rsf F$ and so $\lnot c\in\mathcal K(\rsf F)$ -- 
    where $c$ is the successor of $0$.
\end{proof}

Thus in the discrete case,  the set of filters we are interested in 
are exactly the principal ones and so the collection of these filters
is isomorphic to $\mathcal L$.

Hence we will assume that $\mathcal L$ is non-discrete.

\section{The non-discrete Case}
In this case we are really considering filters as cuts and we will 
set up the usual equivalence of cuts between $]p, 1]$ and $[p, 1]$. 
We cannot get a good theory if we take all cuts but things work well 
if we restrict ourselves to cuts with the same kernel. In that case 
we get an MV-algebra using $\sqTo$ and $\bullet^{+}$.

There is really only one axiom that is hard to check and we will work 
on it indirectly. In what follows we will assume that we are working 
in a linearly ordered MV-algebra and all filters have kernel $\Set1$.

\begin{lem}\label{lem:negate}
    $$\rsf F\sqTo]0, 1]=\rsf F^{+}.$$
\end{lem}
\begin{proof}
    \begin{align*}
	x\in\rsf F\sqTo]0, 1] & \text{ iff }\forall g\notin]0, 1]\ 
	x\to x\notin\rsf F  \\
	 & \text{ iff }x\to0=\lnot x\notin\rsf F  \\
	 & \text{ iff }x\in\rsf F^{+}.
    \end{align*}
\end{proof}

\begin{lem}
    Let $\rsf F_{1}\subseteq\rsf F_{2}$ and $\card{\rsf 
    F_{2}\setminus\rsf F_{1}}\geq 2$. Then
    $$
    \rsf F_{1}\sqTo\rsf F_{2}\not=\Set1.
    $$
\end{lem}
\begin{proof}
    Let $f_{1}<f_{2}\in\rsf F_{2}\setminus\rsf F_{1}$. Then 
    $f_{2}\to f_{1}\in\rsf F_{1}\sqTo\rsf F_{2}$ as if $g\notin\rsf 
    F_{2}$ then $g<f_{1}$ and so
    $(f_{2}\to f_{1})\to g\le (f_{2}\to f_{1})\to f_{1}= 
    f_{2}\notin\rsf F_{1}$. As $f_{2}\to f_{1}<1$ we have the result.
\end{proof}

\begin{prop}
    Let $\rsf F_{1}\subseteq\rsf F_{2}$ and $\card{\rsf 
    F_{2}\setminus\rsf F_{1}}\geq 2$. Then for any filter $\rsf G$ 
    containing $\rsf F_{2}$ 
    $$
    \rsf F_{1}\sqTo\rsf G\not=\rsf F_{2}\sqTo\rsf G.
    $$
\end{prop}
\begin{proof}
    We know that $\rsf F_{2}\sqTo\rsf G\subseteq\rsf F_{1}\sqTo\rsf 
    G$ so we seek an element of the latter that is not in the former.
    
    We can assume that $\rsf G\not=]0, 1]$ -- as in that case we have 
    $\rsf F_{1}\sqTo\rsf G=\rsf F_{1}^{+}\not=\rsf F_{2}^{+}=\rsf 
    F_{2}\sqTo\rsf G$,  by \lemref{lem:negate}.
    
    Let $a_{1}<a_{2}$ in $\rsf F_{2}\setminus\rsf F_{1}$. Now for any 
    $g\in\rsf G\setminus\rsf F_{2}$ we have $a_{i}\to g\in\rsf 
    F_{1}\sqTo\rsf G$ -- since if $h\notin\rsf G$ then we have 
    $h<g$ and so $(a_{i}\to g)\to h< (a_{i}\to g)\to g= 
    a_{i}\notin\rsf F_{1}$. 
    
    We know that $a_{2}\to a_{1}<1$ is not in $\mathcal K(\rsf G)$ 
    and so we can find $g\in\rsf G$ with $g'= g\otimes(a_{2}\to 
    a_{1})\notin\rsf G$. 
    
    If $g'=0$ then pick $0<g''<\rsf G$ and using the fact that 
    $C=[a_{1}, 1]\setminus\rsf F_{1}$ is convex,  then 
    $\Set{g\otimes (z\to a_{1}) | z\in C}$ is also convex,  it 
    contains $0$ and $g= g\otimes(a_{1}\to a_{1})$ and so must 
    contain $g''$. 
    
    Thus we may assume that $a_{2}$ is chosen so that $g'>0$. 
    Now $g\otimes (a_{2}\to a_{1})>0$ implies $(a_{2}\to a_{1})>\lnot 
    g$ and so 
    \begin{align*}
	(a_{2}\to a_{1}) & = (\lnot g)\oplus((a_{2}\to a_{1})\ominus 
	(\lnot g))  \\
	 & = (\lnot g)\oplus((a_{2}\to a_{1})\otimes g)  \\
	 & = g\to g'.
    \end{align*}
    
    We now have $a_{2}\to g\in\rsf F_{1}\sqTo\rsf G$ and we want to 
    show that it is not in $\rsf F_{2}\sqTo\rsf G$.
    In fact we have 
    $a_{1}\le (a_{2}\to g)\to g'$ as 
    \begin{align*}
	a_{1}\le (a_{2}\to g)\to g' & \text{ iff }a_{2}\to g\le 
	a_{1}\to g'  \\
	 & \text{ iff }(a_{1}\to g')\to g\le a_{2}  \\
	 \intertext{ but }
	(a_{1}\to g')\to g & = (g\to g')\to a_{1}  \\
	 & = (a_{2}\to a_{1})\to a_{1}\\
	 &= a_{2}. 
    \end{align*}
\end{proof}

The ``associativity'' axiom is the most difficult to establish. Here
is the easy half.

\begin{lem}\label{lem:FFg}
    Let $\rsf F\subseteq\rsf G$. Then
    $$
    \rsf F\subseteq(\rsf F\sqTo\rsf G)\sqTo\rsf G.
    $$
\end{lem}
\begin{proof}
    It suffices to show that 
    $$
    \rsf F\subseteq\mathcal (\rsf F\sqTo\rsf G)_{a}
    $$
    for all $a\notin\rsf G$.
    
    Let $z\in\rsf F$ and $a\notin\rsf G$ -- and so $a<z$. If $z\to a\in \rsf F\sqTo\rsf G$
    then $z\to a\in\rsf F_{a}$, ie $z\join a= z\notin\rsf F$ -- contradiction.
\end{proof}

\subsection{The principal and coprincipal cases}
It is somewhat important to eliminate the cases $\rsf F=[p,1]$ and $\rsf F=]p,1]$ from consideration
but we still need to know what happens for such filters. 

\begin{lem}
	$$
		]p,1]^+=[\lnot p,1].
	$$
\end{lem}
\begin{proof}
$x\in]p,1]^+$ iff $\lnot x\notin]p,1]$ iff $\lnot x\le p$ iff $\lnot p\le x$ so that 
	$]p,1]^+=[\lnot p,1]$. 
\end{proof}

\begin{prop}\label{prop:prOne}
    Let $0<p<q<1$. Then
    $$
	]q, 1]\sqTo]p, 1]=[q\to p, 1].
    $$
\end{prop}
\begin{proof}
    First we observe that $q\to p$ is actually in. 
    
    If $z<\le p$ then $(q\to p)\to z\le (q\to p)\to p=q$ and so $(q\to 
    p)\to z\notin ]q, 1]$.
    
    If $b<q\to p$ then $q=(q\to p)\to p< b\to p$ and so $b\to 
    p\in ]q, 1]$.
\end{proof}

\begin{prop}\label{prop:prTwo}
    Let $0\le p<q\le 1$. Then
    $$
	[q, 1]\sqTo]p, 1]=]q\to p, 1].
    $$
\end{prop}
\begin{proof}
    First we show that $q\to p$ is not in $[q, 1]\sqTo]p, 1]$ as 
    $p\notin]p, 1]$ but $q=(q\to p)\to p\in [q, 1]$.
    
    If $z>q\to p$ then $z\to p< q=(q\to p)\to p$ and so $z\to 
    p\notin [q, 1]$.
\end{proof}

%
%

\begin{prop}
    Suppose that $0<p\le q<1$. Then 
    $$
    	\vert q, 1]\sqTo[p, 1]=[q\to p, 1]
    $$
    where $\vert q, 1]$ can be either $]q, 1]$ or $[q, 1]$. 
\end{prop}
\begin{proof}
    Let $k<q\to p$. We want to show that there is some $f>q$ with 
    $k\otimes f<p$. 
    
    We note that $p\le q$ implies $q\otimes(q\to p)=p$. Now if every 
    $f>q$ has $p<f\otimes k$ then we get 
    $q\otimes k< q\otimes (q\to p)=p < f\otimes k$,  which implies the
    set $\Set{s\otimes k | s\geq q}$ is not convex. 
    Hence there is some $f>q$ with $f\otimes k=p$. 
    
    Now we can find $q<f'<f$ and so $f'\otimes k<f\otimes k= p$. 
    
%
    
    Finally we have $q\to p\in \vert q, 1]\sqTo[p, 1]$ as if $h<p$
    then $(q\to p)\to h< (q\to p)\to p=q$ and so 
    $(q\to p)\to h\notin\vert q, 1]$.
\end{proof}

\begin{cor}
    If $0<p<1$ then 
    $$
    	]p, 1]\sqTo[p, 1]=\Set1.
    $$
\end{cor}

\subsection{The general case}
Combining several of the above 
results, we get the following theorem on filter equivalence.

\begin{thm}
    $\rsf F\sqTo\rsf G=\Set1$ iff $\rsf G\subseteq\rsf F$ or $\rsf 
    G=[g, 1]$, $\rsf F=]g, 1]$.
\end{thm}
\begin{proof}
    We know that the right side implies the left from 
    \propref{prop:prOne} and \propref{prop:inclOne}.
    
    We also know that if $\rsf F\sqTo\rsf G=\Set1$ then $\card{\rsf 
    G\setminus\rsf F}<2$ and so is either zero (ie $\rsf 
    G\subseteq\rsf F$) or one.
    
    In the latter case if $g\in\rsf G\setminus\rsf F$ then $\rsf 
    G=[g, 1]$ and $\rsf F=]g, 1]$.
\end{proof}

\begin{cor}
    $$
	\rsf F\sqTo\rsf G\not=\Set1\text{ iff }
	\card{\rsf G\setminus\rsf F}\geq 2.
    $$
\end{cor}

\section{Cuts}
Now we are set to define our basic equivalence relation on filters 
and show that this naturally produces an MV-algebra.

\begin{defn}
    Let $\rsf F$ and $\rsf G$ be two filters in $\mathcal L$. Then 
    $$
	\rsf F\equiv\rsf G\text{ iff }\rsf F\sqTo\rsf G= \rsf 
	G\sqTo\rsf F=\Set1.
    $$
\end{defn}

As is usual with cuts (recall $\Q$ giving rise to $\R$) we have 
that
$\rsf F\equiv\rsf G$ iff $\rsf F$ and $\rsf G$ differ by at most one 
point. This is also true here, as we show below. First we establish 
that $\equiv$ is a congruence relation.

\begin{thm}
    $\equiv$ is an equivalence relation on filters.
\end{thm}
\begin{proof}
    Reflexive is given by $\rsf F\sqTo\rsf F=\mathcal K(\rsf F)= \Set1$.
    
    Symmetry is built into the definition.
    
    It remains to prove transitivity,  which we take care of in the 
    following lemma.
\end{proof}

\begin{lem}
    Let $\rsf F\sqTo\rsf G= \rsf G\sqTo\rsf H=\Set1$. Then $\rsf 
    F\sqTo\rsf H=\Set1$. 
\end{lem}
\begin{proof}
	From above we know that $\rsf G\subseteq\rsf F$ or $\rsf 
	G=[g, 1]$, $\rsf F=]g, 1]$.
	
	Now if $\rsf H\subseteq\rsf G$ then either $\rsf H=\rsf G$ in 
	which case $\rsf F\sqTo\rsf H=\rsf F\sqTo\rsf G=\Set1$ or
	$\rsf H\subseteq\rsf F$,  in which case we also have $\rsf 
	F\sqTo\rsf H=\Set1$.
	
	If $\rsf H=[h, 1]$ and $\rsf G=]h, 1]$ then the cases are:
	\begin{enumerate}[{Case }1:]
	    \item $\rsf G=\rsf F$ and then $\rsf F\sqTo\rsf H=\rsf 
	    G\sqTo\rsf H=\Set1$; 
	
	    \item $\rsf G$ is a proper subset of $\rsf F$ in which 
	    case $\rsf H\subseteq\rsf F$; 
	
	    \item $\rsf G=[g, 1]$ and $\rsf F=]g, 1]$. But now $g$ is the successor of $h$ -- 
	    contradicting the fact there are no discrete points.
	\end{enumerate}
\end{proof}

In order to show that $\equiv$ is a 
congruence relation, we need the following lemma.

\begin{lem}
    Let $\rsf F\sqTo\rsf G=\Set1$. Then
    $(\rsf G\sqTo\rsf H)\sqTo(\rsf F\sqTo\rsf H)=\Set1$ for any 
    filter $\rsf H$.
\end{lem}
\begin{proof}
    This is immediate if $\rsf G\subseteq\rsf F$ as we apply 
    \propref{prop:inclOne}.
    
    So suppose that $\rsf G=[g, 1]$, $\rsf F=]g, 1]$.
    
    If $\rsf H=]h, 1]$ we apply propositions \ref{prop:prOne} and 
    \ref{prop:prTwo} to get the result. Hence we may assume that 
    $\rsf H$ is not coprincipal.
    
    We first show that $\rsf F\sqTo\rsf H$ and $\rsf G\sqTo\rsf H$ 
    cannot differ by two points. Suppose that $a_{1}<a_{2}$ are in 
    $\rsf F\sqTo\rsf G$ but not in $\rsf G\sqTo\rsf H$. Thus if 
    $h\notin\rsf H$ then $a_{i}\to h\notin\rsf F$. We know that 
    $a_{i}\notin\rsf G\sqTo\rsf H$ and so there exists $h_{1}, 
    h_{2}\notin\rsf H$ such that $a_{i}\to h_{i}\in\rsf 
    G\setminus\rsf F=\Set{g}$. Hence
    $a_{1}\to h_{1}= a_{2}\to h_{2}= g$. As $a_{1}<a_{2}$ this means 
    that $h_{1}<h_{2}$. Hence 
    $a_{1}\to h_{2}> a_{2}\to h_{2}=g$ and so $a_{1}\to h_{2}\in\rsf 
    F$ -- contradiction. 
    
    If the two sets differ at exactly one point then we must have 
    $\rsf F\sqTo\rsf H=[p, 1]$ for some $p$. We want to show that 
    $p\in\rsf G\sqTo\rsf H$. We have $p\in\rsf F\sqTo\rsf H$ and so 
    $p\to h\notin\rsf F$ for all $h\notin\rsf H$. 
    
    If $p\notin\rsf G\sqTo\rsf H$ then there is some $h\notin\rsf H$ 
    with $p\to h=g$. As $\rsf H$ is not 
    coprincipal  there is some 
    $h<h'<\rsf H$. But now we have 
    $p\to h=g< p\to h'$ and so $p\to h'\in\rsf F$ -- contradiction. 
    
    So this case cannot happen,  and we have $\rsf F\sqTo\rsf H=\rsf 
    G\sqTo\rsf H$. 
\end{proof}

\begin{thm}
    $\equiv$ is a congruence wrt ${\bullet}^{+}$ and 
    $\bullet\sqTo\bullet$.
\end{thm}
\begin{proof}
    As $\rsf F\sqTo\rsf G=\rsf G^{+}\sqTo\rsf F^{+}$ it is immediate 
    that $\rsf F\equiv\rsf G$ implies $\rsf F^{+}\equiv\rsf G^{+}$.
    
    From the lemma we can deduce that $\rsf F\equiv\rsf G$ 
    implies $\rsf G\sqTo\rsf H\equiv \rsf F\sqTo\rsf H$. 
    
    We also have $\rsf H\sqTo\rsf F\equiv\rsf H\sqTo\rsf G$ as 
    $\rsf F^{+}\equiv\rsf G^{+}$ and so
    $\rsf H\sqTo\rsf F= \rsf F^{+}\sqTo\rsf H^{+}\equiv \rsf 
    G^{+}\sqTo\rsf H^{+}= \rsf H\sqTo\rsf G$. 
    
    Finally,  if $\rsf F_{1}\equiv\rsf F_{2}$ and $\rsf 
    G_{1}\equiv\rsf G_{2}$ we have 
    $\rsf F_{1}\sqTo\rsf G_{1}\equiv \rsf F_{1}\sqTo\rsf G_{2} 
    \equiv\rsf F_{2}\sqTo\rsf G_{2}$.
\end{proof}

%

Now we want to know what happens to $(\rsf F\sqTo\rsf G)\sqTo\rsf G$ when $\rsf F\subseteq\rsf G$. 

We know that
$\rsf F\subseteq (\rsf F\sqTo\rsf G)\sqTo\rsf G$ from \lemref{lem:FFg}
and so we have the following corollary.
\begin{cor}
	$$
		\rsf F\sqTo\rsf G= ((\rsf F\sqTo\rsf G)\sqTo\rsf 
		G)\sqTo\rsf G.
	$$
\end{cor}
\begin{proof}
	Taking $\rsf F$ as $\rsf F\sqTo\rsf G$ in the lemma we get LHS$\subseteq$RHS.
	
	As $\rsf F\subseteq(\rsf F\sqTo\rsf G)\sqTo\rsf G$ applying $\sqTo\rsf G$ to both sides reverses
	the inclusion and so RHS$\subseteq$LHS.
\end{proof}

\begin{prop}\label{prop:axiomG}
    Let $\rsf F_{G}=(\rsf F\sqTo\rsf G)\sqTo\rsf G$. Then
    $\rsf F\equiv\rsf F_{G}$.
\end{prop}
\begin{proof}
    We know that $\rsf F\subseteq\rsf F_{G}$ from \lemref{lem:FFg}.
    
    We also know that $\rsf F\sqTo\rsf G= \rsf F_{G}\sqTo\rsf G$. 
    The theorem then implies 
    $\rsf F\sqTo\rsf F_{G}=\Set1$. 
\end{proof}

As usual we define $\le$ by $\rsf F\le\rsf G$ iff $\rsf F\sqTo\rsf 
G=\Set1$. From above we know that 
$\rsf F\le\rsf G$ iff $\rsf G\subseteq\rsf F$ or $\rsf G=[g, 1]$ and 
$\rsf F=]g, 1]$.

\begin{thm}
    Let $P$ be a prime implication filter. 
    Let 
    $$
    	\hat{\mathcal L}_{P}=\brk<\text{PSpec}(P)/\equiv,  
	\bullet^{+},  \sqTo>.
    $$ 
    Then $\hat{\mathcal L}_{P}$ is a linearly ordered MV-algebra. 
\end{thm}
\begin{proof}
    \begin{enumerate}[(a)]
        \item  $\zero=]0, 1]\le\rsf F\le \one=\Set1$ for all $\rsf F$.
    
        \item  $\rsf F\le\rsf G$ iff $\rsf G^{+}\le\rsf F^{+}$ -- by 
	(e).
    
        \item  $\rsf F\sqTo(\rsf H\sqTo\rsf G)= \rsf H\sqTo(\rsf 
	F\sqTo\rsf G)$ by \propref{prop:axiomC}.
    
        \item  $\rsf F\sqTo\rsf F= \mathcal K(\rsf F)= \one$. 
    
        \item  $\rsf F\sqTo\rsf G= \rsf G^{+}\sqTo\rsf F$ by \propref{prop:plus}.
    
        \item  If $\rsf F\sqTo\rsf G= \rsf G\sqTo\rsf F= \one$ then 
	$\rsf F\equiv\rsf G$,  by definition of $\equiv$.
    
        \item  $(\rsf F\sqTo\rsf G)\to\rsf G= (\rsf G\sqTo\rsf 
	F)\sqTo\rsf F$ as if $\rsf F\subseteq\rsf G$ then 
	$(\rsf F\sqTo\rsf G)\to\rsf G\equiv\rsf F$ by \propref{prop:axiomG}. 
	
	And $\rsf G\sqTo\rsf F=\one$ and $\one\to\rsf F=\rsf F$ by 
	propositions \ref{prop:inclOne} and \ref{prop:OneOne}.
    
        \item  $(\rsf F\join\rsf G)\sqTo\rsf H= (\rsf F\sqTo\rsf 
	H)\meet (\rsf G\sqTo\rsf H)$ as we are in a linear order and 
	we have \propref{prop:revIncl}.
    
        \item  $(\rsf F\meet\rsf G)\sqTo\rsf H= (\rsf F\sqTo\rsf 
	H)\join (\rsf G\sqTo\rsf H)$ as we are in a linear order and 
	we have \propref{prop:revIncl}.
    \end{enumerate}
\end{proof}

\subsection{Another look at $\Phi$}
Inside $\hat{\mathcal L}_{P}$ we define $\otimes$ in terms of $\sqTo$ 
and $\bullet^{+}$ as usual in MV-algebras:
$$
\rsf F\otimes\rsf G=(\rsf F\sqTo\rsf G^{+})^{+}.
$$
We noted above (see subsection \ref{secsec:Phi}) that 
$$
\Phi(\rsf F, \rsf G)=\rsf F\otimes\rsf G.
$$
Another natural way to define $\otimes$ is to use the operation of 
$\mathcal L$ directly and let 
$$
T(\rsf F, \rsf G)=\Set{f\otimes g | f\in\rsf F\text{ and }g\in\rsf 
G}\uparrow.
$$

\begin{prop}
    If $\rsf F, \rsf G$ are in $\hat{\mathcal L}_{P}$ then
    $$T(\rsf F, \rsf G)=\rsf F\otimes\rsf G.
    $$
\end{prop}
\begin{proof}
    Consider the ideal 
    $$
    I=\Set{y | (y\oplus\rsf G^{*})\cap\rsf F=\emptyset}.
    $$
    \begin{enumerate}[(i)]
        \item  $I\cap T=\emptyset$. 
	
	If $f\in\rsf F$ and $g\in\rsf G$ then $f\otimes g\not in I$ 
	as $\lnot g\in\rsf G^{*}$ and so
	\begin{align*}
	    (f\otimes g)\oplus\lnot g &= \lnot (\lnot f\oplus\lnot 
	    g)\oplus\lnot g\\
	    &= (f\to\lnot g)\to\lnot g\\
	    &= f\join\lnot g\in\rsf F.
	\end{align*}
    
        \item  If $y\notin I$ then $y\in T$.
	
	If $y\notin I$ then there is some $g\in \rsf G$ such that 
	$y\oplus\lnot g\in\rsf F$. But then we have 
	$g\in\rsf G$, $g\to y= y\oplus\lnot g\in\rsf F$ and so
	$g\otimes(g\to y)\le y$ implies $y\in T$.
    
        \item  $\mathcal L\setminus I= (\rsf F\sqTo\rsf G^{+})^{+}$ as 
	\begin{align*}
	    y\in (\rsf F\sqTo\rsf G^{+})^{+}& \text{ iff }\lnot 
	    y\notin \rsf F\sqTo\rsf G^{+}  \\
	     & \text{ iff }\exists k\notin\rsf G^{+}\ \lnot y\to 
	     k\in\rsf F  \\
	     & \text{ iff }\exists k\ \lnot k\in\rsf G \text{ and 
	     }y\oplus k=\lnot y\to k\in\rsf F \\
	     & \text{ iff } \exists k\  k\in\rsf G^{*} \text{ and 
	     }y\oplus k \in\rsf F \\
	     & \text{ iff }(y\oplus\rsf G^{*})\cap\rsf F\not=\emptyset
	\end{align*}
    \end{enumerate}
    Hence we have 
    $$
    T=\mathcal L\setminus I= (\rsf F\sqTo\rsf G^{+})^{+}= \rsf 
    F\otimes\rsf G.
    $$
\end{proof}

\section{Connecting the pieces}
In this section we consider how the algebras $\hat{\mathcal L}_{P}$
and $\hat{\mathcal L}_{Q}$ are related to one another when 
$P\subseteq Q$. And we consider the connection with the quotient 
$\mathcal L/P$. 

\begin{defn}
    Let $\rsf F$ be a prime lattice filter and $P$ a prime implication 
    filter that properly contains $\mathcal K(\rsf F)$.
    $q_{P}(\rsf F)$
    is the unique $P$-coset $C$ such that $C\cap\rsf 
    F\not=\emptyset\not= C\setminus\rsf F$.
\end{defn}

\begin{thm}
    Let $P\subsetneq Q$ be two prime implication filters. Then
    $$
	\hat\eta_{PQ}\colon\hat{\mathcal L}_{P}\to\mathcal L/Q
    $$
    defined by 
    $$
	\hat\eta_{PQ}(\rsf F)= q_{Q}(\rsf F)
    $$
    is an MV-morphism.
\end{thm}
\begin{proof}
    It is enough to do the case when $\mathcal L=\mathcal L/P$.
    
    First we show that $q(\rsf F^{+})=\lnot q(\rsf F)$.
    
    Let $x\in\rsf F\cap q(\rsf F)$. Then $\lnot x\notin\rsf F^{+}$ as 
    $\lnot^{2}x=x\in\rsf F$. But $\lnot x\in\lnot q(\rsf F)$ and so
    $\lnot q(\rsf F)\setminus\rsf F^{+}\not=\emptyset$.
    
    If $x\in q(\rsf F)\setminus\rsf F$ then $\lnot x\in\rsf F^{+}$
    and so $\lnot q(\rsf F)\cap\rsf F^{+}\not=\emptyset$.
    
    Now we want to show that 
    $$
	q(\rsf F\sqTo\rsf G)=q(\rsf F)\to q(\rsf G).
    $$
    We saw in the proof that $\mathcal K(\rsf F\sqTo\rsf G)=\mathcal 
    K(\rsf F)$ that if 
    $b\in q(\rsf G)\setminus\rsf G$ and $a\in q(\rsf F)\cap F$ that 
    $a\to b\notin\rsf F\sqTo\rsf G$. 
    Also if 
    $b'\in q(\rsf G)\cap\rsf G$ and $a'\in q(\rsf F)\setminus\rsf F$ 
    then $a'\to b'\in\rsf F\sqTo\rsf G$. 
    
    Hence the coset $q(\rsf F)\to\rsf G$ crosses the boundary of 
    $\rsf F\sqTo\rsf G$ and so must be $q(\rsf F\sqTo\rsf G)$.
\end{proof}

Note that we also have an MV-morphism
$\iota_{P}\colon \mathcal L/P\to\hat{\mathcal L}_{P}$ defined by
$$
    \iota_{P}([a]_{P})= P_{a}
$$

\begin{thm}
    If $\mathcal L/P$ is non-discrete then 
    $\iota_{P}$ is an injective MV-morphism.
\end{thm}
\begin{proof}
    We know that $P_{a}=P_{b}$ iff $\eta_{P}(a)=\eta_{P}(b)$ and so $\iota$
    must be injective.
    
    From \propref{prop:quot} and \propref{prop:prOne} we know that 
    $P_{a}\sqTo P_{b}= \eta_{P}^{-1}[P_{a}/P\sqTo P_{b}/P]=
    \eta_{P}^{-1}[][a_{P}], \one]\sqTo][b]_{P}, \one]]= 
    \eta_{P}^{-1}[[[a\to b]_{P}, \one]]\equiv \eta_{P}[][a\to b]_{P}, 
    \one]]$. Hence $\iota$ preserves $\sqTo$.
    
    And $P_{a}^{+}= \eta_{P}^{-1}[][a]_{P}, \one]^{+}]= 
    \eta_{P}^{-1}[[[\lnot a]_{P}, \one]]\equiv \eta_{P}^{-1}[][\lnot 
    a]_{P}, \one]]= P_{\lnot a}$ and so $\iota$ preserves negation.
\end{proof}

And finally we put the pieces together and see that $\hat{\mathcal 
L}_{P}$ interpolates $\eta_{PQ}\colon\mathcal L/P\to\mathcal L/Q$.

\begin{thm}
The composite mapping
$$
    \mathcal L/P\rTo^{\iota_{P}}\hat{\mathcal 
    L}_{P}\rTo^{\hat\eta_{PQ}}\mathcal L/Q
$$
is just
$[a]_{P}\mapsto [a]_{Q}$. 
\end{thm}
\begin{proof}
    It suffices to show that $P\subsetneq Q$ implies 
    $q_{Q}(P_{a})=[a]_{Q}$. Let $q\in Q\setminus P$ and so there is 
    some $f\notin P_{a}$ such that $q\to f\in P_{a}$. 
    
    We know that $P_{a}/P=][a]_{P}, \one]$ and so $f\le_{P} a$. Hence 
    $f\le_{Q} a$. As $q\to f\in P_{a}$ we have $a<_{P}q\to f$ and so
    $a\le_{Q} q\to f$. 
    
    As $[f]_{Q}=[q\to f]_{Q}$ we have $[f]_{Q}=[a]_{Q}$ is the 
    $Q$-boundary coset. 
\end{proof}

\end{document}